\documentclass[11pt]{amsart}
\usepackage{amssymb,amscd,psfrag,graphicx}

\newtheorem{theorem}{Theorem}[section]
\newtheorem{lemma}[theorem]{Lemma}
\newtheorem{corollary}[theorem]{Corollary}
\newtheorem{proposition}[theorem]{Proposition}

\theoremstyle{definition}
\newtheorem{definition}[theorem]{Definition}
\newtheorem{example}[theorem]{Example}
\newtheorem{remark}[theorem]{Remark}
\newtheorem{notation}[theorem]{Notation}
\numberwithin{equation}{theorem}

\def\rad{\operatorname{rad}}
\def\link{\operatorname{link}}

\def\Cdot{\check C^\bullet}
\def\funcF{\operatorname{F}}
\def\funcG{\operatorname{G}}

\def\Ass{\operatorname{Ass}}

\def\Hom{\operatorname{Hom}}
\def\Ext{\operatorname{Ext}}
\def\Tor{\operatorname{Tor}}

\def\bsf{{\boldsymbol{f}}}
\def\bsg{{\boldsymbol{g}}}
\def\bsu{{\boldsymbol{u}}}
\def\bsx{{\boldsymbol{x}}}
\def\bzero{{\boldsymbol{0}}}

\def\fraka{\mathfrak{a}}
\def\frakb{\mathfrak{b}}
\def\frakm{\mathfrak{m}}
\def\frakn{\mathfrak{n}}

\def\NN{\mathbb{N}}
\def\PP{\mathbb{P}}
\def\QQ{\mathbb{Q}}
\def\RR{\mathbb{R}}
\def\ZZ{\mathbb{Z}}
\def\calO{\mathcal O}

\def\ge{\geqslant}
\def\le{\leqslant}
\def\phi{\varphi}
\def\bar{\overline}
\def\tilde{\widetilde}
\def\dlim{\varinjlim}
\def\to{\longrightarrow}
\def\mapsto{\longmapsto}
\renewcommand{\mod}{\,\operatorname{mod}\,}

\begin{document}
\title{Bockstein homomorphisms in local cohomology} 

\author{Anurag K. Singh}
\address{Department of Mathematics, University of Utah, 155 South 1400 East, Salt Lake City, UT~84112, USA} \email{singh@math.utah.edu}

\author{Uli Walther}
\address{Department of Mathematics, Purdue University, 150 N. University Street, West Lafayette, IN~47907, USA} \email{walther@math.purdue.edu}
\thanks{A.K.S.~was supported by NSF grants DMS~0600819 and DMS~0608691.}
\thanks{U.W.~was supported by NSF grant DMS~0555319 and by NSA grant H98230-06-1-0012.}

\date{\today}


\begin{abstract}
Let $R$ be a polynomial ring in finitely many variables over the integers, and fix an ideal $\fraka$ of $R$. We prove that for all but finitely prime integers $p$, the Bockstein homomorphisms on local cohomology, $H^k_\fraka(R/pR)\to H^{k+1}_\fraka(R/pR)$, are zero. This provides strong evidence for Lyubeznik's conjecture which states that the modules $H^k_\fraka(R)$ have a finite number of associated prime ideals.
\end{abstract}
\maketitle

%%%%%%%%%%%%%%%%%%%%%%%%%%%%%%%%%%%%%%%%%%%%%%%%%%%%%%%%%%%%%%%%%%%%%%%%
\section{Introduction}
%%%%%%%%%%%%%%%%%%%%%%%%%%%%%%%%%%%%%%%%%%%%%%%%%%%%%%%%%%%%%%%%%%%%%%%%

Let $R$ be a polynomial ring in finitely many variables over $\ZZ$, the ring of integers. Fix an ideal $\fraka$ of $R$. For each prime integer $p$, applying the local cohomology functor $H^\bullet_\fraka(-)$ to
\[
\CD
0@>>>R/pR@>p>>R/p^2R@>>>R/pR@>>>0\,,
\endCD
\]
one obtains a long exact sequence; the connecting homomorphisms in this sequence are the \emph{Bockstein homomorphisms} for local cohomology,
\[
\beta^k_p\colon H^k_\fraka(R/pR)\to H^{k+1}_\fraka(R/pR)\,.
\]
We prove that for all but finitely many prime integers $p$, the Bockstein homomorphisms $\beta^k_p$ are zero, Theorem~\ref{thm:main}. 

Our study here is motivated by Lyubeznik's conjecture \cite[Remark~3.7]{Lyubeznik:Invent} which states that for regular rings $R$, each local cohomology module $H^k_\fraka(R)$ has finitely many associated prime ideals. This conjecture has been verified for regular rings of positive characteristic by Huneke and Sharp \cite{HS:TAMS}, and for regular local rings of characteristic zero as well as unramified regular local rings of mixed characteristic by Lyubeznik \cite{Lyubeznik:Invent,Lyubeznik:Comm}. It remains unresolved for polynomial rings over $\ZZ$, where it implies that for fixed $\fraka\subseteq R$, the Bockstein homomorphisms $\beta^k_p$ are zero for almost all prime integers $p$; Theorem~\ref{thm:main} provides strong supporting evidence for Lyubeznik's conjecture.

The situation is quite different when, instead of regular rings, one considers hypersurfaces. In Example~\ref{example:hypersurface} we present a hypersurface $R$ over $\ZZ$, with ideal $\fraka$, such that the Bockstein homomorphism $H^2_\fraka(R/pR)\to H^3_\fraka(R/pR)$ is nonzero for each prime integer $p$.

Huneke \cite[Problem~4]{Huneke:Sundance} asked whether local cohomology modules of Noetherian rings have finitely many associated prime ideals. The answer to this is negative: in \cite{Singh:MRL} the first author constructed an example where, for $R$ a hypersurface, $H^3_\fraka(R)$ has $p$-torsion elements for each prime integer $p$, and hence has infinitely many associated primes; see also Example~\ref{example:hypersurface}. The issue of $p$-torsion is central in studying Lyubeznik's conjecture for finitely generated algebras over $\ZZ$, and the Bockstein homomorphism is a first step towards understanding $p$-torsion.

For local or graded rings $R$, the first examples of local cohomology modules $H^k_\fraka(R)$ with infinitely many associated primes were produced by Katzman \cite{Katzman}; these are not integral domains. Subsequently, Singh and Swanson \cite{SS:IMRN} constructed families of graded hypersurfaces $R$ over arbitrary fields, for which a local cohomology module $H^k_\fraka(R)$ has infinitely many associated primes; these hypersurfaces are unique factorization domains that have rational singularities in the characteristic zero case, and are $F$-regular in the case of positive characteristic.

In Section~\ref{sec:preliminary} we establish some properties of Bockstein homomorphisms that are used in Section~\ref{sec:main} in the proof of the main result, Theorem~\ref{thm:main}. Section~\ref{sec:examples} contains various examples, and Section~\ref{sec:SR} is devoted to Stanley-Reisner rings: for $\Delta$ a simplicial complex, we relate Bockstein homomorphisms on reduced simplicial cohomology groups $\tilde H^\bullet(\Delta,\ZZ/p\ZZ)$ and Bockstein homomorphisms on local cohomology modules $H^\bullet_\fraka(R/pR)$, where $\fraka$ is the Stanley-Reisner ideal of $\Delta$. We use this to construct nonzero Bockstein homomorphisms on $H^\bullet_\fraka(R/pR)$, for $R$ a polynomial ring over $\ZZ$.

%%%%%%%%%%%%%%%%%%%%%%%%%%%%%%%%%%%%%%%%%%%%%%%%%%%%%%%%%%%%%%%%%%%%%%%%
\section{Bockstein homomorphisms}
\label{sec:preliminary}
%%%%%%%%%%%%%%%%%%%%%%%%%%%%%%%%%%%%%%%%%%%%%%%%%%%%%%%%%%%%%%%%%%%%%%%%

\begin{definition}
Let $R$ be a commutative Noetherian ring, and $M$ an $R$-module. Let $p$ be an element of $R$ that is a nonzerodivisor on $M$.

Let $\funcF^\bullet$ be an $R$-linear covariant $\delta$-functor on the category of $R$-modules. The exact sequence
\[
\CD
0@>>>M@>p>>M@>>>M/pM@>>>0
\endCD
\]
then induces an exact sequence
\[
\CD
\funcF^k(M/pM)@>{\delta^k_p}>>\funcF^{k+1}(M)@>p>>\funcF^{k+1}(M)@>{\pi^{k+1}_p}>>\funcF^{k+1}(M/pM)\,.
\endCD
\]
The \emph{Bockstein homomorphism} $\beta^k_p$ is the composition
\[
\pi^{k+1}_p\circ\delta^k_p\colon\funcF^k(M/pM)\to\funcF^{k+1}(M/pM)\,.
\]

It is an elementary verification that $\beta^\bullet_p$ agrees with the connecting homomorphisms in the cohomology exact sequence obtained by applying $\funcF^\bullet$ to the exact sequence 
\[
\CD
0@>>>M/pM@>p>>M/p^2M@>>>M/pM@>>>0\,.
\endCD
\]
\end{definition}

Let $\fraka$ be an ideal of $R$, generated by elements $f_1,\dots,f_t$. The covariant $\delta$-functors of interest to us are local cohomology $H^\bullet_\fraka(-)$ and Koszul cohomology $H^\bullet(f_1,\dots,f_t;-)$ discussed next.

Setting $\bsf^e=f_1^e,\dots,f_t^e$, the Koszul cohomology modules $H^\bullet(\bsf^e;M)$ are the cohomology modules of the Koszul complex $K^\bullet(\bsf^e;M)$. For each $e\ge 1$, one has a map of complexes
\[
K^\bullet(\bsf^e;M)\to K^\bullet(\bsf^{e+1};M)\,,
\]
and thus a filtered direct system $\left\{K^\bullet(\bsf^e;M)\right\}_{e\ge1}$. The direct limit of this system can be identified with the \v Cech complex $\Cdot(\bsf;M)$ displayed below:
\[
0\to M\to\bigoplus_iM_{f_i}\to\bigoplus_{i<j}M_{f_if_j}\to\cdots\to M_{f_1 \cdots f_t}\to 0\,.
\]
The local cohomology modules $H^\bullet_\fraka(M)$ may be computed as the cohomology modules of $\Cdot(\bsf;M)$, or equivalently as direct limits of the Koszul cohomology modules $H^\bullet(\bsf^e;M)$. There is an isomorphism of functors
\[
H^k_\fraka(-)\ \cong\ \dlim_e H^k(\bsf^e;-)\,,
\]
and each element of $H^k_\fraka(M)$ lifts to an element of $H^k(\bsf^e;M)$ for $e\gg 0$.
 
\begin{remark}
\label{rem:natural-transformation} 
Let $\funcF^\bullet$ and $\funcG^\bullet$ be covariant $\delta$-functors on the category of $R$-modules, and let $\tau\colon\funcF^\bullet\to\funcG^\bullet$ be a natural transformation. Since the Bockstein homomorphism is defined as the composition of a connecting homomorphism and reduction mod $p$, one has a commutative diagram
\[
\CD
\funcF^k(M/pM)@>>>\funcF^{k+1}(M/pM)\\
@V{\tau}VV @VV{\tau}V\\
\funcG^k(M/pM)@>>>\funcG^{k+1}(M/pM)\,,
\endCD
\]
where the horizontal maps are the respective Bockstein homomorphisms. The natural transformations of interest to us are
\[
H^\bullet(\bsf^e;-)\to H^\bullet(\bsf^{e+1};-)
\] 
and
\[
H^\bullet(\bsf^e;-)\to H^\bullet_\fraka(-)\,,
\]
where $\fraka=(f_1,\dots,f_t)$.
\end{remark}

Let $M$ be an $R$-module, let $p$ be an element of $R$ that is a nonzerodivisor on $M$, and let $\bsf=f_1,\dots,f_t$ and $\bsg=g_1,\dots,g_t$ be elements of $R$ such that $f_i\equiv g_i\mod p$ for each $i$. One then has isomorphisms
\[
H^\bullet(\bsf;M/pM)\cong H^\bullet(\bsg;M/pM)\,,
\]
though the Bockstein homomorphisms on $H^\bullet(\bsf;M/pM)$ and $H^\bullet(\bsg;M/pM)$ may not respect these isomorphisms; see Example~\ref{example:different-lifts}. A key point in the proof of Theorem~\ref{thm:main} is Lemma~\ref{lemma:different-lifts}, which states that upon passing to the direct limits $\dlim_eH^\bullet(\bsf^e;M/pM)$ and $\dlim_eH^\bullet(\bsg^e;M/pM)$, the Bockstein homomorphisms commute with the isomorphisms
\[
\dlim_eH^\bullet(\bsf^e;M/pM)\cong\dlim_eH^\bullet(\bsg^e;M/pM)\,.
\]

\begin{example}
\label{example:different-lifts}
Let $p$ be a nonzerodivisor on $R$. Let $x$ be an element of $R$. The Bockstein homomorphism on Koszul cohomology $H^\bullet(x;R/pR)$ is
\begin{align*}
(0:_{R/pR}x)=H^0(x;R/pR)&\ \to\ H^1(x;R/pR)=R/(p,x)R\\
r\mod(p) &\ \mapsto\ rx/p\mod (p,x)\,.
\end{align*}
Let $y$ be an element of $R$ with $x\equiv y\mod p$. Comparing the Bockstein homomorphisms $\beta,\beta'$ on $H^\bullet(x;R/pR)$ and $H^\bullet(y;R/pR)$ respectively, one sees that the diagram
\[
\CD
\left(0:_{R/pR}x\right)@>\beta>>R/(p,x)R\\
@|@|\\
\left(0:_{R/pR}y\right)@>\beta'>>R/(p,y)R
\endCD
\]
does not commute if $rx/p$ and $ry/p$ differ modulo the ideal $(p,y)R$; for an explicit example, take $R=\ZZ[w,x,z]/(wx-pz)$ and $y=x+p$. Then $\beta(\bar{w})=\bar{z}$, whereas $\beta'(\bar{w})=\bar{z}+\bar{w}$.

Nonetheless, the diagram below does commute, hinting at Lemma~\ref{lemma:different-lifts}.
\[
\CD
H^0(x;R/pR)@>\beta>>H^1(x;R/pR)@>x>>H^1(x^2;R/pR)\\
@| @. @|\\
H^0(y;R/pR)@>\beta'>>H^1(y;R/pR)@>y>>H^1(y^2;R/pR)
\endCD
\]
\end{example}

\begin{lemma}
\label{lemma:different-lifts}
Let $M$ be an $R$-module, and let $p$ be an element of $R$ that is $M$-regular. Suppose $\fraka$ and $\frakb$ are ideals of $R$ with $\rad(\fraka+pR)=\rad(\frakb+pR)$. Then there exists a commutative diagram 
\[
\CD 
\cdots@>>>H^k_\fraka(M/pM)@>>>H^{k+1}_\fraka(M/pM)@>>>\cdots\\
@. @VVV @VVV\\
\cdots@>>>H^k_\frakb(M/pM)@>>>H^{k+1}_\frakb(M/pM)@>>>\cdots
\endCD
\]
where the horizontal maps are the respective Bockstein homomorphisms, and the vertical maps are natural isomorphisms.
\end{lemma}

\begin{proof}
It suffices to consider the case $\fraka=\frakb+yR$, where $y\in\rad(\frakb+pR)$. For each $R$-module $N$, one has an exact sequence 
\[
\CD
@>>>H^{k-1}_\frakb(N)_y@>>>H^k_\fraka(N)@>>>H^k_\frakb(N)@>>>H^k_\frakb(N)_y@>>>
\endCD
\]
which is functorial in $N$; see for example \cite[Exercise~14.4]{magnumopus}. Using this for
\[
\CD
0@>>>M@>p>>M@>>>M/pM@>>>0\,,
\endCD
\]
one obtains the commutative diagram below, with exact rows and columns.
\[
\CD
H^{k-1}_\frakb(M/pM)_y\!\!@>>>\!\!H^k_\frakb(M)_y@>p>>H^k_\frakb(M)_y\!\!@>>>\!\!H^k_\frakb(M/pM)_y\\
@VVV @VVV @VVV @VVV\\
H^k_\fraka(M/pM)\!\!@>>>\!\!H^{k+1}_\fraka(M)@>p>>H^{k+1}_\fraka(M)\!\!@>>>\!\!H^{k+1}_\fraka(M/pM)\\
@V\theta^kVV @VVV @VVV @VV\theta^{k+1}V\\
H^k_\frakb(M/pM)\!\!@>>>\!\!H^{k+1}_\frakb(M)@>p>>H^{k+1}_\frakb(M)\!\!@>>>\!\!H^{k+1}_\frakb(M/pM)\\
@VVV @VVV @VVV @VVV\\
H^k_\frakb(M/pM)_y\!\!@>>>\!\!H^{k+1}_\frakb(M)_y@>p>>H^{k+1}_\frakb(M)_y\!\!@>>>\!\!H^{k+1}_\frakb(M/pM)_y
\endCD
\]

\medskip\noindent Since $H^\bullet_\frakb(M/pM)$ is $y$-torsion, it follows that $H^\bullet_\frakb(M/pM)_y=0$. Hence the maps $\theta^\bullet$ are isomorphisms, and the desired result follows.
\end{proof}

%%%%%%%%%%%%%%%%%%%%%%%%%%%%%%%%%%%%%%%%%%%%%%%%%%%%%%%%%%%%%%%%%%%%%%%%
\section{Main Theorem}
\label{sec:main}
%%%%%%%%%%%%%%%%%%%%%%%%%%%%%%%%%%%%%%%%%%%%%%%%%%%%%%%%%%%%%%%%%%%%%%%%

\begin{theorem}
\label{thm:main}
Let $R$ be a polynomial ring in finitely many variables over the ring of integers. Let $\fraka=(f_1,\dots,f_t)$ be an ideal of $R$. 

If a prime integer $p$ is a nonzerodivisor on Koszul cohomology $H^{k+1}(\bsf;R)$, then the Bockstein homomorphism $H^k_\fraka(R/pR)\to H^{k+1}_\fraka(R/pR)$ is zero.

In particular, the Bockstein homomorphisms on $H^\bullet_\fraka(R/pR)$ are zero for all but finitely many prime integers $p$.
\end{theorem}

We use the following notation in the proof, and also later in Section~\ref{sec:SR}.

\begin{notation}
\label{notation}
Let $R$ be a ring with an endomorphism $\phi$. Set $R^\phi$ to be the $R$-bimodule with $R$ as its underlying Abelian group, the usual action of $R$ on the left, and the right $R$-action defined by $r'r=\phi(r)r'$ for $r\in R$ and $r'\in R^\phi$. One thus obtains a functor
\[
M\mapsto R^\phi\otimes_R M
\]
on the category of $R$-modules, where $R^\phi\otimes_RM$ is viewed as an $R$-module via the left $R$-module structure on $R^\phi$.
\end{notation}

\begin{proof}[Proof of Theorem~\ref{thm:main}]
The $R$-modules $H^k(\bsf;R)$ are finitely generated, so
\[
\bigcup_k\Ass H^k(\bsf;R)
\]
is a finite set of prime ideals. These finitely many prime ideals contain finitely many prime integers, so the latter assertion follows from the former.

Fix a prime $p$ that is a nonzerodivisor on $H^{k+1}(\bsf;R)$. Suppose $R=\ZZ[x_1,\dots,x_n]$, set $\psi$ to be the endomorphism of $R$ with $\psi(x_i)=x_i^p$ for each~$i$. For each positive integer $e$, consider $R^{\psi^e}$ as in Notation~\ref{notation}. The module $R^{\psi^e}$ is $R$-flat, so applying $R^{\psi^e}\otimes_R(-)$ to the injective homomorphism
\[
\CD
H^{k+1}(\bsf;R)@>p>>H^{k+1}(\bsf;R)
\endCD
\]
one obtains an injective homomorphism 
\[
\CD
H^{k+1}(\psi^e(\bsf);R)@>p>>H^{k+1}(\psi^e(\bsf);R)\,,
\endCD
\]
where $\psi^e(\bsf)=\psi^e(f_1),\dots,\psi^e(f_t)$. Thus, the connecting homomorphism in the exact sequence
\[
\CD
H^k(\psi^e(\bsf);R/pR)@>\delta>>H^{k+1}(\psi^e(\bsf);R)@>p>>H^{k+1}(\psi^e(\bsf);R)
\endCD
\]
is zero, and hence so is the Bockstein homomorphism
\begin{equation}
\label{eqn:zeromap}
\CD
H^k(\psi^e(\bsf);R/pR)@>>>H^{k+1}(\psi^e(\bsf);R/pR)\,.
\endCD
\end{equation}

The families of ideals
\[
\big\{\big(\psi^e(\bsf)\big)R/pR\big\}_{e\ge1}\qquad\text{ and }\qquad\big\{\fraka^eR/pR\big\}_{e\ge1}
\]
are cofinal, so 
\[
H^k_\fraka(R/pR)\cong\dlim_e H^k(\psi^e(\bsf);R/pR)\,.
\]
Let $\eta$ be an element of $H^k_\fraka(R/pR)$. There exists an integer $e$ and an element $\tilde{\eta}\in H^k(\psi^e(\bsf);R/pR)$ such that $\tilde{\eta}\mapsto\eta$. By Remark~\ref{rem:natural-transformation} and Lemma~\ref{lemma:different-lifts}, one has a commutative diagram
\[
\CD
H^k(\psi^e(\bsf);R/pR)@>>>H^{k+1}(\psi^e(\bsf);R/pR)\\
@VVV @VVV \\
H^k_\fraka(R/pR)@>>>H^{k+1}_\fraka(R/pR)\,,
\endCD
\]
where the map in the upper row is zero by \eqref{eqn:zeromap}. It follows that $\eta$ maps to zero in $H^{k+1}_\fraka(R/pR)$.
\end{proof}

%%%%%%%%%%%%%%%%%%%%%%%%%%%%%%%%%%%%%%%%%%%%%%%%%%%%%%%%%%%%%%%%%%%%%%%%
\section{Examples}
\label{sec:examples}
%%%%%%%%%%%%%%%%%%%%%%%%%%%%%%%%%%%%%%%%%%%%%%%%%%%%%%%%%%%%%%%%%%%%%%%%

Example~\ref{example:ramified} shows that the Bockstein $\beta^0_p\colon H^0_\fraka(R/pR)\to H^1_\fraka(R/pR)$ need not be zero for $R$ a regular ring. In Example~\ref{example:hypersurface} we exhibit a hypersurface $R$ over $\ZZ$, with ideal $\fraka$, such that $\beta^2_p\colon H^2_\fraka(R/pR)\to H^3_\fraka(R/pR)$ is nonzero for each prime integer $p$. Example~\ref{example:elliptic} is based on elliptic curves, and includes an intriguing open question.

\begin{example}
\label{example:ramified}
Let $\fraka=(f_1,\dots,f_t)\subseteq R$ and let $[r]\in H^0_\fraka(R/pR)$. There exists an integer $n$ and $a_i\in R$ such that $rf_i^n=pa_i$ for each $i$. Using the \v Cech complex on $\bsf$ to compute $H^\bullet_\fraka(R/pR)$, one has
\[
\beta^0_p([r])=\big[\big(a_1/f_1^n,\dots,a_t/f_t^n\big)\big]\in H^1_\fraka(R/pR)\,.
\]

For an example where $\beta^0_p$ is nonzero, take $R=\ZZ[x,y]/(xy-p)$ and $\fraka=xR$. Then $[y]\in H_{xR}^0(R/pR)$, and
\[
\beta^0_p([y])=[1/x]\in H_{xR}^1(R/pR)\,.
\]
We remark that $R$ is a regular ring: since $R_x=\ZZ[x,1/x]$ and $R_y=\ZZ[y,1/y]$ are regular, it suffices to observe that the local ring $R_{(x,y)}$ is also regular. Note, however, that $R$ is \emph{ramified} since $R_{(x,y)}$ is a ramified regular local ring. 
\end{example}

\begin{example}
\label{example:hypersurface}
We give an example where $\beta^2_p$ is nonzero for each prime integer $p$; this is based on \cite[Section~4]{Singh:MRL} and \cite{Singh:Guanajuato}. Consider the hypersurface
\[
R=\ZZ[u,v,w,x,y,z]/(ux+vy+wz)
\]
and ideal $\fraka=(x,y,z)R$. Let $p$ be an arbitrary prime integer. Then the element $(u/yz,-v/xz,w/xy)\in R_{yz}\oplus R_{xz}\oplus R_{xy}$ gives a cohomology class
\[
\eta=[(u/yz,-v/xz,w/xy)]\ \in\ H_\fraka^2(R/pR)\,.
\]
It is easily seen that $\beta^2_p(\eta)=0$; we verify below that $\beta^2_p(F(\eta))$ is nonzero, where $F$ denotes the Frobenius action on $H_\fraka^2(R/pR)$. Indeed, if
\[
\beta^2\big(F(\eta)\big)=\left[\frac{(ux)^p+(vy)^p+(wz)^p}{p(xyz)^p}\right]
\]
is zero, then there exists $k\in\NN$ and elements $c_i\in R/pR$ such that
\begin{equation}
\label{eqn:hypersurface}
\frac{(ux)^p+(vy)^p+(wz)^p}{p}(xyz)^{k}=c_1x^{p+k}+c_2y^{p+k}+c_3z^{p+k}
\end{equation}
in $R/pR$. Assign weights as follows:
\begin{align*}
x\colon (1,0,0,0), \qquad\qquad\qquad & u\colon (-1,0,0,1), \\
y\colon (0,1,0,0), \qquad\qquad\qquad & v\colon (0,-1,0,1), \\
z\colon (0,0,1,0), \qquad\qquad\qquad & w\colon (0,0,-1,1). 
\end{align*}
There is no loss of generality in taking the elements $c_i$ to be homogeneous, in which case $\deg c_1=(-p,k,k,p)$, so $c_1$ must be a scalar multiple of $u^py^kz^k$. Similarly, $c_2$ is a scalar multiple of $v^pz^kx^k$ and $c_3$ of $w^px^ky^k$. Hence
\[
\frac{(ux)^p+(vy)^p+(wz)^p}{p}(xyz)^k\ \in\ (xyz)^k\big(u^px^p,\ v^py^p,\ w^pz^p\big)R/pR\,.
\]
Canceling $(xyz)^k$ and specializing $u\mapsto 1, v\mapsto 1,w \mapsto 1$, we have
\[
\frac{x^p+y^p+(-x-y)^p}{p}\ \in\ \big(x^p,y^p\big)\ZZ/p\ZZ[x,y]\,,
\]
which is easily seen to be false.
\end{example}

\begin{example}
\label{example:elliptic}
Let $E\subset\PP^2_\QQ$ be a smooth elliptic curve. Set $R=\ZZ[x_0,\dots,x_5]$ and let $\fraka\subset R$ be an ideal such that $(R/\fraka)\otimes_\ZZ\QQ$ is the homogeneous coordinate ring of $E\times\PP^1_\QQ$ for the Segre embedding $E\times\PP^1_\QQ\subset\PP^5_\QQ$. For all but finitely many prime integers $p$, the reduction mod $p$ of $E$ is a smooth elliptic curve $E_p$, and $R/(\fraka+pR)$ is a homogeneous coordinate ring for $E_p\times\PP^1_{\ZZ/p}$; we restrict our attention to such primes.

The elliptic curve $E_p$ is said to be \emph{ordinary} if the Frobenius map
\[
H^1(E_p,\calO_{E_p})\to H^1(E_p,\calO_{E_p}) 
\]
is injective, and is \emph{supersingular} otherwise. By well-known results on elliptic curves \cite{Deuring, Elkies}, there exist infinitely many primes $p$ for which $E_p$ is ordinary, and infinitely many for which $E_p$ is supersingular. The local cohomology module $H^4_\fraka(R/pR)$ is zero if $E_p$ is supersingular, and nonzero if $E_p$ is ordinary, see \cite[page~75]{HS}, \cite[page~219]{Lyubeznik:Compositio}, \cite[Corollary~2.2]{SW:Contemp}, or \cite[Section~22.1]{magnumopus}. It follows that the multiplication by $p$ map
\[
\CD
H^4_\fraka(R)@>p>>H^4_\fraka(R)
\endCD
\]
is surjective for infinitely many primes $p$, and not surjective for infinitely many $p$. Lyubeznik's conjecture implies that this map is injective for almost all primes $p$, equivalently that the connecting homomorphism
\[
\CD
H^3_\fraka(R/pR)@>{\delta}>>H^4_\fraka(R)
\endCD
\]
is zero for almost all $p$. We do not know if this is true. However, Theorem~\ref{thm:main} implies that $\beta^3_p$, i.e., the composition of the maps
\[
\CD
H^3_\fraka(R/pR)@>{\delta}>>H^4_\fraka(R)@>\pi>>H^4_\fraka(R/pR)\,,
\endCD
\]
is zero for almost all $p$. It is known that the ideal $\fraka R/pR$ is generated up to radical by four elements, \cite[Theorem~1.1]{SW:Contemp}, and that it has height $3$. Hence $\beta^k_p=0$ for $k\neq 3$.
\end{example}

%%%%%%%%%%%%%%%%%%%%%%%%%%%%%%%%%%%%%%%%%%%%%%%%%%%%%%%%%%%%%%%%%%%%%%%%
\section{Stanley-Reisner rings}
\label{sec:SR}
%%%%%%%%%%%%%%%%%%%%%%%%%%%%%%%%%%%%%%%%%%%%%%%%%%%%%%%%%%%%%%%%%%%%%%%%

Bockstein homomorphisms originated in algebraic topology where they were used, for example, to compute the cohomology rings of lens spaces. In this section, we work in the context of simplicial complexes and associated Stanley-Reisner ideals, and relate Bockstein homomorphisms on simplicial cohomology groups to those on local cohomology modules; see Theorem~\ref{thm:top-alg}. We use this to construct nonzero Bockstein homomorphisms $H^k_\fraka(R/pR)\to H^{k+1}_\fraka(R/pR)$, for $R$ a polynomial ring over $\ZZ$.

\begin{definition}
Let $\Delta$ be a simplicial complex with vertices $1,\dots,n$. Set $R$ to be the polynomial ring $\ZZ[x_1,\dots,x_n]$. The \emph{Stanley-Reisner ideal} of $\Delta$~is
\[
\fraka=(\bsx^\sigma\mid \sigma\not\in \Delta)R\,,
\]
i.e., $\fraka$ is the ideal generated by monomials $\bsx^\sigma=\prod_{i=1}^n x_i^{\sigma_i}$ such that $\sigma$ is not a face of $\Delta$. In particular, if $\{i\}\notin\Delta$, then $x_i\in\fraka$.

The ring $R/\fraka$ is the \emph{Stanley-Reisner ring} of $\Delta$.
\end{definition}

\begin{example}
\label{example:rp2a}
Consider the simplicial complex corresponding to a triangulation of the real projective plane $\RR\PP^2$ depicted in Figure~\ref{fig:rp2}.
\begin{figure}[htbp]
\[
\psfrag{1}{$1$}
\psfrag{2}{$2$}
\psfrag{3}{$3$}
\psfrag{4}{$4$}
\psfrag{5}{$5$}
\psfrag{6}{$6$}
\includegraphics[width=5cm]{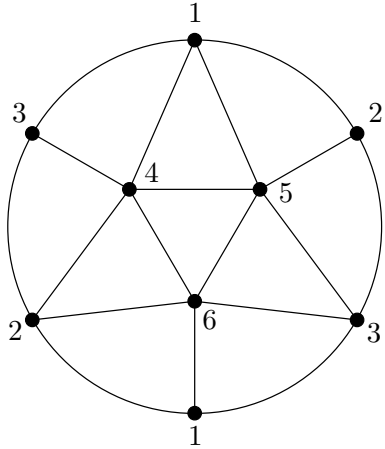}
\]
\caption{A triangulation of the real projective plane}
\label{fig:rp2}
\end{figure}

The associated Stanley-Reisner ideal in $\ZZ[x_1,\dots,x_6]$ is generated by\begin{multline*}
x_1x_2x_3\,,\ \ x_1x_2x_4\,,\ \ x_1x_3x_5\,,\ \ x_1x_4x_6\,,\ \ x_1x_5x_6\,,\ \ x_2x_3x_6\,,\\
x_2x_4x_5\,,\ \ x_2x_5x_6\,,\ \ x_3x_4x_5\,,\ \ x_3x_4x_6\,.
\end{multline*}
\end{example}

\begin{remark}
\label{rem:Hochster}
Let $\Delta$ be a simplicial complex with vertex set $\{1,\dots,n\}$. Let $\fraka$ be the associated Stanley-Reisner ideal in $R=\ZZ[x_1,\dots,x_n]$, and set $\frakn$ to be the ideal $(x_1,\dots,x_n)$. The ring $R$ has a $\ZZ^n$-grading where $\deg x_i$ is the $i$-th unit vector; this induces a grading on the ring $R/\fraka$, and also on the \v Cech complex $\Cdot=\Cdot(\bsx;R/\fraka)$. Note that a module
\[
(R/\fraka)_{x_{i_1}\cdots x_{i_k}}
\]
is nonzero precisely if $x_{i_1}\cdots x_{i_k}\notin\fraka$, equivalently $\{i_1,\dots,i_k\}\in\Delta$. Hence ${[\Cdot]}_\bzero$, the $(0,\dots,0)$-th graded component of $\Cdot$, is the complex that computes the reduced simplicial cohomology $\tilde{H}^\bullet(\Delta;\ZZ)$, with the indices shifted by one. This provides natural identifications
\begin{equation}
\label{eqn:alg-top-zero}
\left[H^k_\frakn(R/\fraka)\right]_{\bzero}=\tilde{H}^{k-1}(\Delta;\ZZ)\qquad\text{ for }k\ge0\,.
\end{equation}
Similarly, for $p$ a prime integer, one has
\[
\left[H^k_\frakn(R/(\fraka+pR))\right]_{\bzero}=\tilde{H}^{k-1}(\Delta;\ZZ/p\ZZ)\,,
\]
and an identification of Bockstein homomorphisms
\[
\CD
\tilde{H}^{k-1}(\Delta;\ZZ/p\ZZ)@>\beta>>\tilde{H}^k(\Delta;\ZZ/p\ZZ)\\
@| @|\\
\left[H^k_\frakn(R/(\fraka+pR))\right]_{\bzero}@>\beta>>\left[H^{k+1}_\frakn(R/(\fraka+pR))\right]_{\bzero}\,.
\endCD
\]
Proposition~\ref{prop:Hochster} extends these natural identifications.
\end{remark}

\begin{definition}
Let $\Delta$ be a simplicial complex, and let $\tau$ be a subset of its vertex set. The \emph{link} of $\tau$ in $\Delta$ is the set
\[
\link_\Delta(\tau)=\{\sigma\in\Delta\mid\sigma\cap\tau=\emptyset\ \text{ and }\ \sigma\cup\tau\in\Delta\}\,.
\]
\end{definition}

\begin{proposition}
\label{prop:Hochster}
Let $\Delta$ be a simplicial complex with vertex set $\{1,\dots,n\}$, and let $\fraka$ in $R=\ZZ[x_1,\dots,x_n]$ be the associated Stanley-Reisner ideal.

Let $G$ be an Abelian group. Given $\bsu\in\ZZ^n$, set $\tilde{\bsu}=\{i\mid u_i<0\}$. Then
\[
{H^k_\frakn(R/\fraka\otimes_{\ZZ}G)}_\bsu=
\begin{cases}
\tilde{H}^{k-1-|\tilde{\bsu}|}(\link_{\Delta}(\tilde{\bsu});G) &\text{ if $\bsu\le\bzero$}\,,\\
0 &\text{ if $u_j>0$ for some $j$\,,}
\end{cases}
\]
where $\frakn=(x_1,\dots,x_n)$. Moreover, for $\bsu\le\bzero$, there is a natural identification of Bockstein homomorphisms
\[
\CD
\left[H^k_\frakn(R/(\fraka+pR))\right]_{\bsu}@>\beta>>\left[H^{k+1}_\frakn(R/(\fraka+pR))\right]_{\bsu}\\
@| @|\\
\tilde{H}^{k-1-|\tilde{\bsu}|}(\link_{\Delta}(\tilde{\bsu});\ZZ/p\ZZ)@>\beta>>\tilde{H}^{k-|\tilde{\bsu}|}(\link_{\Delta}(\tilde{\bsu});\ZZ/p\ZZ)\,.
\endCD
\]
\end{proposition}

This essentially follows from Hochster \cite{Hochster}, though we sketch a proof next; see also \cite[Section 5.3]{BH}. Note that $\tilde{\bzero}=\emptyset$ and $\link_{\Delta}(\emptyset)=\Delta$, so one recovers \eqref{eqn:alg-top-zero} by setting $\bsu=\bzero$ and $G=\ZZ$.

\begin{proof}
We may assume $G$ is nontrivial; set $T=R/\fraka\otimes_\ZZ G$. As in Remark~\ref{rem:Hochster}, we compute $H^k_\frakn(T)$ as the cohomology of the \v Cech complex $\Cdot=\Cdot(\bsx;T)$.

We first consider the case where $u_j>0$ for some $j$. If the module
\[
\left[T_{x_{i_1}\cdots x_{i_k}}\right]_\bsu
\]
is nonzero, then $\bar{x}_j\neq0$ in $T_{x_{i_1}\cdots x_{i_k}}$, and so $\{j,i_1,\dots,i_k\}\in\Delta$. Hence, if the complex $[\Cdot]_\bsu$ is nonzero, then it computes---up to index shift---the reduced simplicial cohomology of a cone, with $j$ the cone vertex. It follows that $[H^k_\frakn(T)]_\bsu=0$ for each $k$.

Next, suppose $\bsu\le\bzero$. Then the module $[T_{x_{i_1}\cdots x_{i_k}}]_\bsu$ is nonzero precisely if $\{i_1,\dots,i_k\}\in\Delta$ and $\tilde{\bsu}\subseteq\{i_1,\dots,i_k\}$. Hence, after an index shift of $|\tilde{\bsu}|+1$, the complex ${[\Cdot]}_\bsu$ agrees with a complex $C^\bullet(\link_{\Delta}(\tilde{\bsu});G)$ that computes the reduced simplicial cohomology groups $\tilde{H}^\bullet(\link_{\Delta}(\tilde{\bsu});G)$.

The assertion about Bockstein maps now follows, since the complexes
\[
\minCDarrowwidth15pt
\CD
0@>>>[\Cdot(\bsx;R/\fraka)]_\bsu@>p>>[\Cdot(\bsx;R/\fraka)]_\bsu@>>>[\Cdot(\bsx;R/(\fraka+pR))]_\bsu@>>>0
\endCD
\]
and
\[
\minCDarrowwidth15pt
\CD
0@>>>C^\bullet(\link_{\Delta}(\tilde{\bsu});\ZZ)@>p>>C^\bullet(\link_{\Delta}(\tilde{\bsu});\ZZ)@>>>C^\bullet(\link_{\Delta}(\tilde{\bsu});\ZZ/p)@>>>0
\endCD
\]
agree after an index shift.
\end{proof}

Thus far, we have related Bockstein homomorphisms on reduced simplicial cohomology groups to those on $H^\bullet_\frakn(R/(\fraka+pR))$. Our interest, however, is in the Bockstein homomorphisms on $H^\bullet_\fraka(R/pR)$. Towards this, we need the following duality result:

\begin{proposition}
\label{proposition:duality}
Let $(S,\frakm)$ be a Gorenstein local ring. Set $d=\dim S$, and let~$(-)^\vee$ denote the functor $\Hom_S(-,E)$, where $E$ is the injective hull of~$S/\frakm$. Suppose $p\in S$ is a nonzerodivisor on $S$ as well as a nonzerodivisor on a finitely generated $S$-module $M$. Then there are natural isomorphisms
\begin{small}
\[
\minCDarrowwidth12pt
\CD
{\Ext^{k+1}_S(M,S/pS)}^\vee@>>>{\Ext^{k+1}_S(M,S)}^\vee@>p>>{\Ext^{k+1}_S(M,S)}^\vee@>>>{\Ext^k_S(M,S/pS)}^\vee\\
@VV\cong V @VV\cong V @VV\cong V @VV\cong V\\
H^{d-k-2}_\frakm(M/pM)@>>>H^{d-k-1}_\frakm(M)@>p>>H^{d-k-1}_\frakm(M)@>>>H^{d-k-1}_\frakm(M/pM)\,,
\endCD
\]
\end{small}
where the top row originates from applying ${\Hom_S(M,-)}^\vee$ to the sequence
\[
\CD
0@>>>S@>p>>S@>>>S/pS@>>>0\,,
\endCD
\]
and the bottom row from applying $H^0_\frakm(-)$ to the sequence
\[
\CD
0@>>>M@>p>>M@>>>M/pM@>>>0\,.
\endCD
\]
\end{proposition}

\begin{proof}
Let $F_\bullet$ be a free resolution of $M$. The top row of the commutative diagram in the proposition is the homology exact sequence of
\begin{small}
\[
\CD
0@<<<\Hom_S(F_\bullet,S)^\vee@<p<<\Hom_S(F_\bullet,S)^\vee@<<<\Hom_S(F_\bullet,S/pS)^\vee@<<<0\\
@. @| @| @| \\
0@<<<F_\bullet\otimes_SE@<p<<F_\bullet\otimes_SE@<<<F_\bullet\otimes_SE_p@<<<0
\endCD
\]
\end{small}
where $E_p=\Hom_S(S/pS,E)$.

Let $\Cdot$ be the \v Cech complex on a system of parameters of~$S$. Since $S$ is Gorenstein, $\Cdot$ is a flat resolution of $H^d(\Cdot)=E$, and therefore
\begin{equation}
\label{eqn:flat}
H_k(F_\bullet\otimes_SE)=\Tor^S_k(M,E)=H_k(M\otimes_S\Cdot)=H^{d-k}_\frakm(M)\,.
\end{equation}

Since $p$ is $M$-regular, the complex $F_\bullet/pF_\bullet$ is a resolution of $M/pM$ by free $S/pS$-modules. Hence
\[
F_\bullet\otimes_SE_p=F_\bullet\otimes_S(S/pS)\otimes_{S/pS}E_p=(F_\bullet/pF_\bullet)\otimes_{S/pS}E_p\,.
\]
Repeating the proof of~\eqref{eqn:flat} over the Gorenstein ring $S/pS$, which has dimension $d-1$, we see that
\[
H_k(F_\bullet\otimes_SE_p)=H^{d-1-k}_\frakm(M/pM)\,.\qedhere
\]
\end{proof}

\begin{remark}
\label{rem:union}
Let $R=\ZZ[x_1,\dots,x_n]$ be a polynomial ring. Fix an integer $t\ge2$, and set $\phi$ to be the endomorphism of $R$ with $\phi(x_i)=x_i^t$ for each~$i$; note that $\phi$ is flat. Consider $R^\phi$ as in Notation~\ref{notation}; the functor $\Phi$ with
\[
\Phi(M)=R^\phi\otimes_R M
\]
is an exact functor $\Phi$ on the category of $R$-modules. There is an isomorphism $\Phi(R)\cong R$ given by $r'\otimes r\mapsto\phi(r)r'$. More generally, for $M$ a free $R$-module, one has $\Phi(M)\cong M$. For a map $\alpha$ of free modules given by a matrix $(\alpha_{ij})$, the map $\Phi(\alpha)$ is given by the matrix $(\phi(\alpha_{ij}))$. Since $\Phi$ takes finite free resolutions to finite free resolutions, it follows that for $R$-modules $M$ and $N$, one has natural isomorphisms
\[
\Phi\big(\Ext^k_R(M,N)\big)\cong\Ext^k_R\big(\Phi(M),\Phi(N)\big)\,,
\]
see \cite[\S\,2]{Lyubeznik:Compositio} or \cite[Remark~2.6]{SW:IJM}.

Let $\fraka$ be an ideal generated by square-free monomials. Since $\phi(\fraka)\subseteq\fraka$, there is an induced endomorphism $\bar{\phi}$ of $R/\fraka$. The image of $\bar{\phi}$ is spanned by those monomials in $x_1^t,\dots,x_d^t$ that are not in $\fraka$. Using the map that is the identity on these monomials, and kills the rest, one obtains a splitting of $\bar{\phi}$. It follows that the endomorphism $\bar{\phi}\colon R/\fraka\to R/\fraka$ is pure.

Since the family of ideals $\{\phi^e(\fraka)R\}$ is cofinal with the family $\{\fraka^e\}$, the module $H^k_\fraka(R)$ is the direct limit of the system
\[
\Ext^k_R(R/\fraka,R)\to\Phi\big(\Ext^k_R(R/\fraka,R)\big)\to\Phi^2\big(\Ext^k_R(R/\fraka,R)\big)\to\cdots\,.
\]
The maps above are injective; see \cite[Theorem~1]{Lyubeznik:monomial}, \cite[Theorem~1.1]{Mustata}, or \cite[Theorem~1.3]{SW:IJM}. Similarly, one has injective maps in the system
\[
\dlim_e\Phi^e\big(\Ext^k_R(R/\fraka,R/pR)\big)\cong H^k_\fraka(R/pR)\,,
\]
and hence a commutative diagram with injective rows and exact columns:
\begin{small}
\[
\minCDarrowwidth20pt
\CD
\Ext^k_R(R/\fraka,R/pR)@>>>\Phi\big(\Ext^k_R(R/\fraka,R/pR)\big)@>>>\cdots@>>>H^k_\fraka(R/pR)\\
@VVV @VVV @. @VVV\\
\Ext^{k+1}_R(R/\fraka,R)@>>>\Phi\big(\Ext^{k+1}_R(R/\fraka,R)\big)@>>>\cdots@>>>H^{k+1}_\fraka(R)\\
@VpVV @VpVV @. @VVpV\\
\Ext^{k+1}_R(R/\fraka,R)@>>>\Phi\big(\Ext^{k+1}_R(R/\fraka,R)\big)@>>>\cdots@>>>H^{k+1}_\fraka(R)\\
@VVV @VVV @. @VVV\\
\Ext^{k+1}_R(R/\fraka,R/pR)@>>>\Phi\big(\Ext^{k+1}_R(R/\fraka,R/pR)\big)@>>>\cdots@>>>H^{k+1}_\fraka(R/pR)
\endCD
\]
\end{small}

It follows that the vanishing of the Bockstein homomorphism
\[
H^k_\fraka(R/pR)\to H^{k+1}_\fraka(R/pR)
\]
is equivalent to the vanishing of the Bockstein homomorphism
\[
\Ext^k_R(R/\fraka,R/pR)\to\Ext^{k+1}_R(R/\fraka,R/pR)\,.
\]
\end{remark}

\begin{theorem}
\label{thm:top-alg}
Let $\Delta$ be a simplicial complex with vertices $1,\dots,n$. Set $R=\ZZ[x_1,\dots,x_n]$, and let $\fraka\subseteq R$ be the Stanley-Reisner ideal of~$\Delta$. For each prime integer $p$, the following are equivalent:
\begin{enumerate}
\item the Bockstein $H^k_\fraka(R/pR)\to H^{k+1}_\fraka(R/pR)$ is zero;

\item the Bockstein homomorphism
\[
\tilde{H}^{n-k-2-|\tilde{\bsu}|}(\link_{\Delta}(\tilde{\bsu});\ZZ/p\ZZ)\to\tilde{H}^{n-k-1-|\tilde{\bsu}|}(\link_{\Delta}(\tilde{\bsu});\ZZ/p\ZZ)
\]
is zero for each $\bsu\in\ZZ^n$ with $\bsu\le\bzero$.
\end{enumerate}
\end{theorem}

Setting $\bsu=\bzero$ immediately yields:

\begin{corollary}
\label{cor:top-alg}
If the Bockstein homomorphism
\[
\tilde{H}^j(\Delta;\ZZ/p\ZZ)\to\tilde{H}^{j+1}(\Delta;\ZZ/p\ZZ)
\]
is nonzero, then so is the Bockstein homomorphism
\[
H^{n-j-2}_\fraka(R/pR)\to H^{n-j-1}_\fraka(R/pR)\,.
\]
\end{corollary}

\begin{proof}[Proof of Theorem~\ref{thm:top-alg}]
By Remark~\ref{rem:union}, condition (1) is equivalent to the vanishing of the Bockstein homomorphism
\[
\label{eqn:ext}
\Ext^k_R(R/\fraka,R/pR)\to\Ext^{k+1}_R(R/\fraka,R/pR)\,.
\]
Set $\frakm=(p,x_1,\ldots,x_n)$. Using Proposition~\ref{proposition:duality} for the Gorenstein local ring $R_\frakm$, this is equivalent to the vanishing of the Bockstein homomorphism
\[
H^{n-k-1}_\frakm(R/(\fraka+pR))\to H^{n-k}_\frakm(R/(\fraka+pR))\,,
\]
which, by Lemma~\ref{lemma:different-lifts}, is equivalent to the vanishing of the Bockstein
\[
H^{n-k-1}_\frakn(R/(\fraka+pR))\to H^{n-k}_\frakn(R/(\fraka+pR))\,,
\]
where $\frakn=(x_1,\ldots,x_n)$. Proposition~\ref{prop:Hochster} now completes the proof.
\end{proof}

\begin{example}
\label{example:rp2b}
Let $\Delta$ be the triangulation of the real projective plane $\RR\PP^2$ from Example~\ref{example:rp2a}, and $\fraka$ the corresponding Stanley-Reisner ideal. Let $p$ be a prime integer. We claim that the Bockstein homomorphism
\begin{equation}
\label{eqn:rp2bock}
H^3_\fraka(R/pR)\to H^4_\fraka(R/pR)
\end{equation}
is nonzero if and only if $p=2$.

For the case $p=2$, first note that the cohomology groups in question are
\begin{align*}
 \tilde{H}^0(\RR\PP^2;\ZZ)&=0\,,
&\tilde{H}^1(\RR\PP^2;\ZZ)&=0\,,
&\tilde{H}^2(\RR\PP^2;\ZZ)&=\ZZ/2\,,\\
 \tilde{H}^0(\RR\PP^2;\ZZ/2)&=0\,,
&\tilde{H}^1(\RR\PP^2;\ZZ/2)&=\ZZ/2\,,
&\tilde{H}^2(\RR\PP^2;\ZZ/2)&=\ZZ/2\,,
\end{align*}
so $0\to\ZZ\overset{2}\to\ZZ\to\ZZ/2\to0$ induces the exact sequence
\begin{small}
\[
\minCDarrowwidth15pt
\CD
0@>>>\tilde{H}^1(\RR\PP^2;\ZZ/2)@>\delta>>\tilde{H}^2(\RR\PP^2;\ZZ)@>2>>\tilde{H}^2(\RR\PP^2;\ZZ)@>\pi>>\tilde{H}^2(\RR\PP^2;\ZZ/2)@>>>0\\
@. @| @| @| @| @.\\
0@>>>\ZZ/2@>\delta>>\ZZ/2@>2>>\ZZ/2@>\pi>>\ZZ/2@>>>0
\endCD
\]
\end{small}

Since $\delta$ and $\pi$ are isomorphisms, so is the Bockstein homomorphism $\tilde{H}^1(\RR\PP^2;\ZZ/2)\to\tilde{H}^2(\RR\PP^2;\ZZ/2)$. By Corollary~\ref{cor:top-alg}, the Bockstein~\eqref{eqn:rp2bock} is nonzero in the case $p=2$.

If $p$ is an odd prime, then $R/(\fraka+pR)$ is Cohen-Macaulay: this may be obtained from Proposition~\ref{prop:Hochster}, or see \cite[page~180]{Hochster}. Hence
\[
H^k_\frakm(R/(\fraka+pR))=0\qquad\text{ for each $k\neq3$ and $p>2$}\,.
\]
By \cite[Theorem~1.1]{Lyubeznik:Compositio} it follows that
\[
H^{6-k}_\fraka(R/pR)=0\qquad\text{ for each $k\neq 3$ and $p>2$}\,,
\]
so the Bockstein homomorphism~\eqref{eqn:rp2bock} must be zero for $p$ an odd prime.

We mention that the arithmetic rank of the ideal $\fraka R/pR$ in $R/pR$ is $4$, independent of the prime characteristic $p$; see \cite[Example~2]{Yan}.
\end{example}

\begin{example}
\label{example:dunce}
Let $\Lambda_m$ be the \emph{$m$-fold dunce cap}, i.e., the quotient of the unit disk obtained by identifying each point on the boundary circle with its translates under rotation by $2\pi/m$; specifically, for each $\theta$, the points
\[
e^{i(\theta+2\pi r/m)}\qquad\text{ for }r=0,\dots,m-1\,,
\]
are identified. The $2$-fold dunce cap $\Lambda_2$ is homeomorphic to the real projective plane from Examples~\ref{example:rp2a}~and~\ref{example:rp2b}.

The complex $0\to\ZZ\overset{m}\to\ZZ\to0$, supported in homological degrees~$1,2$, computes the reduced simplicial homology of $\Lambda_m$. Let $\ell\ge2$ be an integer; $\ell$ need not be prime. The reduced simplicial cohomology groups of $\Lambda_m$ with coefficients in $\ZZ$ and $\ZZ/\ell$ are
\begin{align*}
 \tilde{H}^0(\Lambda_m;\ZZ)&=0\,,
&\tilde{H}^1(\Lambda_m;\ZZ)&=0\,,
&\tilde{H}^2(\Lambda_m;\ZZ)&=\ZZ/m\,,\\
 \tilde{H}^0(\Lambda_m;\ZZ/\ell)&=0\,,
&\tilde{H}^1(\Lambda_m;\ZZ/\ell)&=\ZZ/g\,,
&\tilde{H}^2(\Lambda_m;\ZZ/\ell)&=\ZZ/g\,,
\end{align*}
where $g=\gcd(\ell,m)$. Consequently, $0\to\ZZ\overset{\ell}\to\ZZ\to\ZZ/\ell\to0$ induces the exact sequence
\[
\minCDarrowwidth14pt
\CD
0@>>>\tilde{H}^1(\Lambda_m;\ZZ/\ell)@>\delta>>\tilde{H}^2(\Lambda_m;\ZZ)@>\ell>>\tilde{H}^2(\Lambda_m;\ZZ)@>\pi>>\tilde{H}^2(\Lambda_m;\ZZ/\ell)@>>>0\\
@. @| @| @| @| @.\\
0@>>>\ZZ/g@>\delta>>\ZZ/m@>\ell>>\ZZ/m@>\pi>>\ZZ/g@>>>0
\endCD
\]
The image of $\delta$ is the cyclic subgroup of $\ZZ/m$ generated by the image of $m/g$. Consequently the Bockstein homomorphism
\[
\tilde{H}^1(\Lambda_m;\ZZ/\ell)\to\tilde{H}^2(\Lambda_m;\ZZ/\ell)
\]
is nonzero if and only if $g$ does not divide $m/g$, equivalently, $g^2$ does not divide $m$.

Suppose $m$ is the product of distinct primes $p_1,\dots,p_r$. By the above discussion, the Bockstein homomorphisms $\tilde{H}^1(\Lambda_m;\ZZ/p_i)\to\tilde{H}^2(\Lambda_m;\ZZ/p_i)$ are nonzero. Let $\Delta$ be a simplicial complex corresponding to a triangulation of $\Lambda_m$, and let $\fraka$ in $R=\ZZ[x_1,\dots,x_n]$ be the corresponding Stanley-Reisner ideal. Corollary~\ref{cor:top-alg} implies that the Bockstein homomorphism
\[
H^{n-3}_\fraka(R/p_iR)\to H^{n-2}_\fraka(R/p_iR)
\]
is nonzero for each $p_i$. It follows that the local cohomology module $H^{n-2}_\fraka(R)$ has a $p_i$ torsion element for each $i=1,\dots,r$.
\end{example}

\begin{example}
We record an example where the Bockstein homomorphism $H^k_\fraka(R/pR)\to H^{k+1}_\fraka(R/pR)$ is zero though $H^{k+1}_\fraka(R)$ has $p$-torsion; this torsion is detected by ``higher'' Bockstein homomorphisms 
\[
H^k_\fraka(R/p^eR)\to H^{k+1}_\fraka(R/p^eR)\,,
\]
i.e., those induced by $0\to\ZZ\overset{p^e}\to\ZZ\to\ZZ/p^e\to0$.

Consider the $4$-fold dunce cap $\Lambda_4$. It follows from Example~\ref{example:dunce} that the Bockstein homomorphism
\[
\ZZ/2=\tilde{H}^1(\Lambda_4;\ZZ/2)\to\tilde{H}^2(\Lambda_4;\ZZ/2)=\ZZ/2
\]
is zero, whereas the Bockstein homomorphism
\[
\ZZ/4=\tilde{H}^1(\Lambda_4;\ZZ/4)\to\tilde{H}^2(\Lambda_4;\ZZ/4)=\ZZ/4
\]
is nonzero. Let $\fraka$ in $R=\ZZ[x_1,\dots,x_9]$ be the Stanley-Reisner ideal corresponding to the triangulation of $\Lambda_4$ depicted in Figure~\ref{fig:dunce}. While we have restricted to $p$-Bockstein homomorphisms, corresponding results may be derived for $p^e$-Bockstein homomorphisms; it then follows that the Bockstein homomorphism $H^6_\fraka(R/2R)\to H^7_\fraka(R/2R)$ is zero, whereas the Bockstein $H^6_\fraka(R/4R)\to H^7_\fraka(R/4R)$ is nonzero.

\begin{figure}[htbp]
\[
\psfrag{1}{$1$}
\psfrag{2}{$2$}
\psfrag{3}{$3$}
\psfrag{4}{$4$}
\psfrag{5}{$5$}
\psfrag{6}{$6$}
\psfrag{7}{$7$}
\psfrag{8}{$8$}
\psfrag{9}{$9$}
\includegraphics[width=6cm]{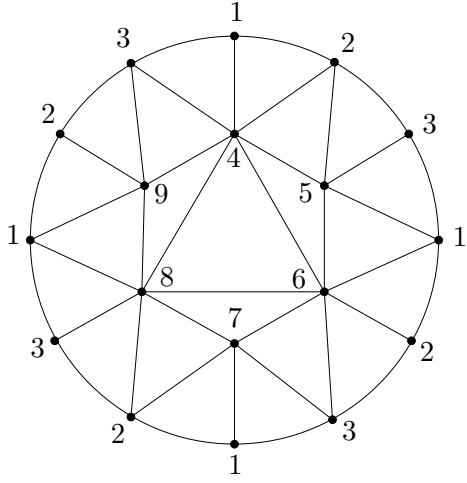}
\]
\caption{A triangulation of the $4$-fold dunce cap}
\label{fig:dunce}
\end{figure}
\end{example}

Given finitely many prime integers $p_1,\dots,p_r$, Example~\ref{example:dunce}, provides a polynomial ring $R=\ZZ[x_1,\dots,x_n]$ with monomial ideal $\fraka\subseteq R$ such that, for some $k$, the Bockstein homomorphism
\[
H^{k-1}_\fraka(R/p_iR)\to H^k_\fraka(R/p_iR)
\]
is nonzero for each $p_i$, in particular, $H^k_\fraka(R)$ has nonzero $p_i$-torsion elements. The following theorem shows that for $\fraka$ a monomial ideal, each $H^k_\fraka(R)$ has nonzero $p$-torsion elements for at most finitely many primes $p$.

\begin{theorem}
\label{thm:monomial-finite}
Let $R=\ZZ[x_1,\dots,x_n]$ be a polynomial ring, and $\fraka$ an ideal that is generated by monomials. Then each local cohomology module $H^k_\fraka(R)$ has at most finitely many associated prime ideals. In particular, $H^k_\fraka(R)$ has nonzero $p$-torsion elements for at most finitely many prime integers $p$.
\end{theorem}

\begin{proof}
Consider the $\NN^n$--grading on $R$ where $\deg x_i$ is the $i$-th unit vector. This induces an $\NN^n$-grading on $H^k_\fraka(R)$, and it follows that each associated prime of $H^k_\fraka(R)$ must be $\NN^n$-graded, hence of the form $(x_{i_1},\dots,x_{i_k})$ or $(p,x_{i_1},\dots,x_{i_k})$ for $p$ a prime integer. Thus, it suffices to prove that $H^k_\fraka(R)$ has nonzero $p$-torsion elements for at most finitely many primes $p$.

After replacing $\fraka$ by its radical, assume $\fraka$ is generated by square-free monomials. Fix an integer $t\ge2$ and, as in Remark~\ref{rem:union}, let $\phi$ be the endomorphism of $R$ with $\phi(x_i)=x_i^t$ for each~$i$. Then
\[
H^k_\fraka(R/pR)\cong\dlim_e\Phi^e\big(\Ext^k_R(R/\fraka,R)\big)\,,
\]
where the maps in the direct system are injective. It suffices to verify that $M$ has nonzero $p$-torsion if and only if $\Phi(M)$ has nonzero $p$-torsion; this is indeed the case since $\Phi$ is an exact functor.
\end{proof}

%%%%%%%%%%%%%%%%%%%%%%%%%%%%%%%%%%%%%%%%%%%%%%%%%%%%%%%%%%%%%%%%%%%%%%%%

\end{document}